\documentclass[12pt, twoside]{article}
\usepackage{amsmath,amsthm,amssymb}
\usepackage{times}
\usepackage{enumerate}
\usepackage{bm}
\usepackage[all]{xy}
\usepackage{mathrsfs}
\usepackage{amscd}
\usepackage{color}
\usepackage{subeqnarray}
\usepackage{cases}
\def\titlerunning#1{\gdef\titrun{#1}}
\makeatletter
\def\author#1{\gdef\autrun{\def\and{\unskip, }#1}\gdef\@author{#1}}
\def\address#1{{\def\and{\\\hspace*{18pt}}\renewcommand{\thefootnote}{}%
		\footnote {#1}}%
	\markboth{\autrun}{\titrun}}
\makeatother
\def\email#1{e-mail: #1}
\def\subjclass#1{{\renewcommand{\thefootnote}{}%
		\footnote{\emph{Mathematics Subject Classification (2010):} #1}}}
\def\keywords#1{\par\medskip
	\noindent\textbf{Keywords.} #1}

\newtheorem{theorem}{Theorem}[section]
\newtheorem{corollary}[theorem]{Corollary}
\newtheorem{lemma}[theorem]{Lemma}
\newtheorem{proposition}[theorem]{Proposition}
\theoremstyle{definition}
\newtheorem{definition}[theorem]{Definition}
\newtheorem{remark}[theorem]{Remark}

\numberwithin{equation}{section}

\xyoption{all}

\def \C {\mathbb{C}}

\def \G {\mathcal{G}}

\def \de {\delta}

\def \la {\lambda}

\def\w {\omega}
\def\Om{\Omega}

\def\na {\nabla}

\begin{document}
\baselineskip=17pt

\titlerunning{the complex Yang-Mills equations on closed $4$-manifolds}
\title{Some boundedness properties of solutions to the complex Yang-Mills equations on closed $4$-manifolds}

\author{Teng Huang}

\date{}

\maketitle

\address{T. Huang: School of Mathematical Sciences, University of Science and Technology of China; Key Laboratory of Wu Wen-Tsun Mathematics, Chinese Academy of Sciences, Hefei, 230026, P.R. China; \email{htmath@ustc.edu.cn; htustc@gmail.com}}

\begin{abstract}
In this article, we study the {analytical} properties of the solutions of the complex Yang-Mills equations on a closed Riemannian four-manifold $X$ with a Riemannian metric $g$. The main result is that if $g$ is $good$ and the connection is an approximate ASD connection, then the extra field has a positive lower {bound}. As an application, we obtain some gap results for Yang-Mills connections and Kapustin-Witten equations. 
\end{abstract}
\section{Introduction}

Let $X$ be a smooth, $n$-dimensional manifold {endowed} with a smooth Riemannian metric, $g$, $P$ be a principal $G$-bundle over $X$. The structure group $G$ is assumed to be a compact Lie group with Lie algebra $\mathfrak{g}$.  {We denote by $A$ a smooth connection on $P$. The curvature of $A$ is a $2$-form $F_{A}$ with values in $\mathfrak{g}_{P}$,where  $\mathfrak{g}_{P}:=P\times_{G}\mathfrak{g}$ is the adjoint bundle of $P$.}  We define by $d_{A}$ the exterior covariant derivative on {sections} of $\Om^{\bullet}(X,\mathfrak{g}_{P}):=\Om^{\bullet}T^{\ast}X\otimes\mathfrak{g}$ with respect to the connection $A$. We denote by $\phi$ a one-form {taking} values  in $\mathfrak{g}_{P}$. The curvature $\mathcal{F}_{\mathcal{A}}$ of the complex connection ${\mathcal{A}}:=A+\sqrt{-1}\phi$ is a two-form with values in  $\mathfrak{g}_{P}^{\C}:=P\times_{G}(\mathfrak{g}\otimes{\C})$ \cite[Equation (2.1)]{Gagliardo/Uhlenbeck:2012}:
$$\mathcal{F}_{\mathcal{A}}:=[(d_{A}+\sqrt{-1}\phi)\wedge(d_{A}+\sqrt{-1}\phi) ]=F_{A}-\phi\wedge\phi+\sqrt{-1}d_{A}\phi.$$
{We call that the connection $\mathcal{A}$ is  $complex$ it only means that the connection $\mathcal{A}$ is a $1$-form with value in $\mathfrak{g}^{\mathbb{C}}_{P}$ but the base manifold is always real Riemannian manifold.} The complex Yang-Mills functional is defined in any dimension as the norm squared of the complex curvature:
$$YM_{\C}(A,\phi):=\frac{1}{2}\int_{X}|\mathcal{F}_{\mathcal{A}}|^{2}=\frac{1}{2}\int_{X}(|F_{A}-\phi\wedge\phi|^{2}+|d_{A}\phi|^{2}).$$
A short calculation shows that the gradient of the complex Yang-Mills energy functional $YM_{\C}$ with respect to the $L^{2}$-metric on $C^{\infty}(X,\Om^{1}(\mathfrak{g}_{P})\oplus\Om^{1}(\mathfrak{g}_{P}))$,
\begin{equation}\nonumber
(\mathcal{M}(A,\phi),(a,b))_{L^{2}(X)}:=\frac{d}{dt}YM_{\C}
(A+ta,\phi+tb)\mid_{t=0}=YM_{\C}'(A,\phi)(a,b),
\end{equation}
for all $(a,b)\in C^{\infty}(X,\Om^{1}(\mathfrak{g}_{P})\oplus\Om^{1}(\mathfrak{g}_{P}))$, is given by
\begin{equation}\nonumber
\begin{split}
(\mathcal{M}(A,\phi),(a,b))_{L^{2}(X)}:&=(d_{A}^{\ast}(F_{A}-\phi\wedge\phi),a)_{L^{2}(X)}+([a,\phi],d_{A}\phi)_{L^{2}(X)}\\
&-([\phi,b],F_{A})_{L^{2}(X)}+(d_{A}^{\ast}d_{A}\phi,b)_{L^{2}(X)}.
\\
\end{split}
\end{equation}
where $d_{A}^{\ast}$ is the $L^{2}$ adjoint of the exterior covariant derivative $d_{A}$ We call $(A,\phi)$ a complex Yang-Mills pair with respect to the Riemannian metric $g$ on $X$, if it is a critical point for $YM_{\C}$, that is, $\mathcal{M}(A,\phi)=0$, i.e., the pair $(A,\phi)$ satisfies
\begin{subequations}\label{E1.1}
\begin{numcases}{}
d^{\ast}_{A}(F_{A}-\phi\wedge\phi)-(-1)^{n}\ast[\phi,\ast d_{A}\phi]=0,\\
d_{A}^{\ast}d_{A}\phi+(-1)^{n}\ast[\phi,\ast(F_{A}-\phi\wedge\phi)]=0.
\end{numcases}
\end{subequations}
These can also be succinctly written as,
\begin{equation}\label{E9}
d^{\ast_{\C}}_{\mathcal{A}}\mathcal{F}_{\mathcal{A}}=0,
\end{equation}
where the Hodge operator $\ast_{\C}$ is the $\C$-linearly extended  of $\ast$. The curvature $\mathcal{F}_{\mathcal{A}}$ of connection $\mathcal{A}$ always obeys the $Bianchi$ $identity$ 
$$d_{\mathcal{A}}\mathcal{F}_{\mathcal{A}}=0.$$
But the complex Yang-Mills equations are not elliptic, even after the real gauge-equivalence, so it is necessary to add the moment map condition \cite{Gagliardo/Uhlenbeck:2012},
$$d_{A}^{\ast}\phi=0.$$
{In this article, we consider the solutions of the complex Yang-Mills equations with the moment map condition.}

Note that the complex Yang-Mills $L^{2}$-energy functional {reduces} to the pure Yang-Mills $L^{2}$-energy functional when $\phi=0$, respectively
$$YM(A):=\frac{1}{2}\int_{X}|F_{A}|^{2}dvol_{g},$$
and $A$ is a Yang-Mills connection with respect to the Riemannian metric $g$ on $X$ if it is a critical point for $YM(A)$ , that is,
$$\mathcal{M}(A)=d_{A}^{\ast}F_{A}=0.$$
For any smooth connection $A$, we recall that 
$$F_{A}=F_{A}^{+}+F_{A}^{-}\in\Om^{2}(X,\mathfrak{g}_{P})=\Om^{2,+}(X,\mathfrak{g}_{P})\oplus\Om^{2,-}(X,\mathfrak{g}_{P}),$$
corresponding to the positive and negative eigenspaces $\Om^{2,\pm}$ of the Hodge star operator $\ast:\Om^{2}\rightarrow\Om^{2}$ defined by the Riemannian metric $g$ so $F_{A}^{\pm}=\frac{1}{2}(1\pm\ast)F_{A}\in\Om^{2,\pm}(X,\mathfrak{g}_{P})$. We denote  $$\mathcal{B}_{\de}(P,g):=\{A:\|F_{A}^{+}\|_{L^{2}(X)}\leq\de\}$$ and
$$\mathcal{B}^{\C}_{\de}(P,g):=\{\mathcal{A}:\|F_{\mathcal{A}}^{+}\|_{L^{2}(X)}\leq\de\}.$$
Suppose that $(A,\phi)$ is a smooth solution of complex Yang-Mills equations. If $A\in\mathcal{B}_{\de}(P,g)$,  then by the Energy identity in Proposition \ref{P4}, $$\|F_{\mathcal{A}}^{+}\|_{L^{2}(X)}\leq \|F_{A}^{+}\|_{L^{2}(X)},$$
i.e., 
$$\mathcal{B}_{\de}\cap\{\mathcal{A}:\mathcal{M}(A,\phi)=0\}\subset\mathcal{B}^{\C}_{\de}\cap\{\mathcal{A}:\mathcal{M}(A,\phi)=0\}.$$
There are many articles {studying} the energy gap of Yang-Mills connection. The motivation of these gap results is partly from physics and partly from math--to better understand the {behaviour} of the Yang-Mills functional near its critical points. On four-dimensional manifold, Feehan \cite{Feehan2014.12} showed that energies associated to  non-minimal  Yang-Mills connections on a principal bundle over a {closed}, four-dimensional, smooth manifold  are separated from the energy of the minimal Yang-Mills connections by a  uniform positive constant depending at most on the Riemannian metric and the Pontrjagin degree of the principal bundle, i.e., the Yang-Mills connections on $\mathcal{B}_{\de}(P,g)$, for a suitable positive constant $\de$, must be ASD connection. For higher dimensional manifolds case, Feehan applied the Lojasiewicz-Simon gradient inequality \cite[Theorem 3.2]{Feehan2015} to {give} a complete proof of energy gap for pure Yang-Mills. The author has given another proof of this theorem without using Lojasiewicz-Simon gradient inequality \cite{HuangCRM}. 

We would like to understand the behaviour of the complex Yang-Mills functional near its critical points. It is an interesting question to consider, whether the complex Yang-Mills equations have the energy gap phenomenon. We study a simplest case that the connection {is} in $\mathcal{B}_{\de}$. The Theorem 1.2 on \cite{HuangSIGMA} states that if the first Pontrjagin number $p_{1}(P)[X]$ is zero and the flat connections on $P$ are all $non$-$degenerate$, then the real curvature $F_{A}$ of the complex Yang-Mills connection should {have a uniform positive lower bound} unless the connection $A$ is flat.

In this article, we study the solutions of complex-Yang-Mills equations on a $G$-bundle over a closed four-manifold. Because there are many potential combinations of conditions on $G$, $P$, $X$ and $g$ which imply that $\operatorname{{Coker}}d_{A}^{+}=0$ when $A$ is anti-self-dual with respect to the Riemannian metric $g$, it is convenient to introduce the $Good$ $Riemannian$ $Metric$  .
\begin{definition}\label{D1}(\cite[Definition 1.3]{Feehan2014.12})
Let $G$ be a compact, simple Lie group, $X$ be a closed, connected, four-dimensional, oriented, smooth manifold, and $\eta\in H^{2}(X;\pi_{1}(G))$ be an obstruction class. We say that a Riemannian metric $g$ on $X$ is $good$ if for every principal $G$-bundle $P$ over $X$ with $\eta(P)=\eta$ and non-positive first Pontrjagin number $p_{1}(P)[X]$ and every smooth connection $A$ on $P$ with $F_{A}^{+}=0$ on $X$, then $\operatorname{Corker}d^{+}_{A}=0$.
\end{definition}	
We denote by $\mu_{g}(A)$ the least eigenvalue of operator $d_{A}^{+}d_{A}^{+,\ast}$ \cite[Definition 3.1]{Taubes1982}. It's easy to see that $\mu_{g}(A)$ is a function with respect to connection $A$ and $\operatorname{Corker}d^{+}_{A}=0$ {is equivalent to} $\mu_{g}(A)>0$. Feehan proved  that $\mu_{g}(A)$  has a positive lower bound with respect to $[A]\in\mathcal{B}_{\varepsilon}(P,g)$ under {assumption that} the Riemannian metric $g$ is $good$.  We then proved that 
\begin{theorem}\label{T1}
Let $G$ be a compact, simple Lie group and P be a principal $G$-bundle over a closed, connected, four-dimensional, oriented, smooth manifold $X$ with Riemannian metric $g$ that is $good$ in the sense of Definition \ref{D1}. Then there exists a positive constant $\varepsilon=\varepsilon(g,P)$ {such that the following holds.} If the pair $(A,\phi)$ is a smooth solution of complex Yang-Mills equations on $P$ such that 
\begin{equation}\nonumber
{	0\neq\|F_{A}^{+}\|_{L^{2}(X)}\leq\varepsilon,}
\end{equation}
then { there exists a positive constant  $C=C(g,P,\varepsilon)$ such that} $$ \|\phi\|_{L^{2}(X)}\geq C.$$
\end{theorem}
\begin{remark}
For a Riemannian metric $g$ on a four-dimensional, oriented manifold $X$ let $R(x)$ denote its scalar curvature at a point $x\in X$ and let $\mathcal{W}^{\pm}(x)$ denote its self-dual and anti-self-dual Weyl curvature tensors at $x$. Define $w^{\pm}:=$Largest eigenvalue of $\mathcal{W}^{\pm}(x)$, $\forall x\in X$.	When $g$ is positive in the sense of \cite[Definition 3.1]{Feehan2014.12}, i.e.,  $\frac{1}{3}R-2\w^{+}>0$ on $X$, then \cite[Lemma 3.3]{Feehan2014.12} implies that $g$ is $good$ metric.  For $generic$ metric $g$ on $X$, suppose that $G$, $P$, and $X$ obey one of the sets of conditions as the following, \\
(1) $b^{+}(X)=0$, $G=SU(2)$ or $SO(3)$,\\
(2) $b^{+}(X)>0$, $G=SO(3)$ and the second Stiefel-Whitney class, $w_{2}(P)\neq 0$,\\
the author also {claimed} that $g$ is $good$ in the sense of Definition \ref{D1} \cite[Corollaries 3.9]{Feehan2014.12}. We mean by $generic$ metric the metrics in a second category subset of the
space of $C^{k}$ metrics for some fixed  $k>2$ \cite[Page 13]{Donaldson}.
\end{remark}
There are many examples of complex Yang-Mills equations. For example, the Kapustin-Witten equations {   }\cite{Gagliardo/Uhlenbeck:2012,Haydy,Kapustin/Witten},
\begin{equation}\label{E20}
(F_{A}-\phi\wedge\phi)^{+}=0,\  (d_{A}\phi)^{-}=d_{A}^{\ast}\phi=0.
\end{equation}
on a closed four-dimension manifolds $X$. In fact, the Kapustin-Witten equations can be seen as a general anti-self-dual $G_{\C}$-connection \cite{Gagliardo/Uhlenbeck:2012}. If $X$ is a closed K\"{a}hler surface, the Kapustin-Witten is equivalent to the Hitchin-Simpson equations \cite{Hitchin,Simpson1988}. From Theorem \ref{T1}, we would also obtain a lower {bound} on extra field of Kapustin-Witten equations. Kapustin and Witten originally obtained the following family of equations parametrized by $\tau$ from a topological twist of $\mathcal{N}=4$ super Yang–Mills theory in 4 dimensions.
\begin{equation}\label{E21}
(F_{A}-(\phi\wedge \phi)+\tau d_{A}\phi)^{+}=0,\ (F_{A}-\phi\wedge\phi-\tau^{-1}d_{A}\phi)^{-}=0,\ d^{\ast}_{A}\phi=0
\end{equation}
The Eqs. (\ref{E20}) are the case for $\tau=0$ in the above. The set of Eqs. (\ref{E21}) for $\tau=\{\infty\}$ is the “orientation reversed” one to the Eqs. (\ref{E20})  \cite[Section 2]{Gagliardo/Uhlenbeck:2012}.
\begin{corollary}\label{C1}(\cite[Theorem 1.1]{HuangLMP})
Assume the hypotheses of Theorem \ref{T1}. Then there exists a positive constant $\varepsilon=\varepsilon(g,P)$ with following significance.\ If the pair $(A,\phi)$ is a smooth solution of Kapustin-Witten equations on $P$ such that 
$$\|\phi\|_{L^{2}(X)}\leq\varepsilon$$
then  $A$ is anti-self-dual with respect to the metric $g$. 
\end{corollary}
When $\phi=0$, the complex Yang-Mills equations are { reduced} to pure Yang-Mills equation. {Following the estimate on Proposition \ref{P3}, we have}
\begin{corollary}\label{C2}(\cite[Theorem 1]{Feehan2014.12})
	Assume the hypotheses of Theorem \ref{T1}. Then there exists a positive constant $\varepsilon=\varepsilon(g,P)$ with following significance.\ If $A$ is a smooth Yang-Mills connection on $P$ with respect to the metric $g$, such that 
	\begin{equation}\nonumber
	\|F_{A}^{+}\|_{L^{2}(X)}\leq\varepsilon
	\end{equation}
	then  $A$ is anti-self-dual with respect to the metric $g$. 
\end{corollary}
The organization of this paper is as follows. In section 2, we first review some estimates of the complex Yang-Mills equations and Chern-Weil formula. In section 3, we proved that $A$ must be a ASD connection if $\|F_{A}^{+}\|_{L^{2}(X)}+\|\phi\|_{L^{2}(X)}$ is sufficiently small and $g$ is $good$ metric, when the $(A,\phi)$ is a $C^{\infty}$ solution of complex Yang-Mills equation. The main idea of this result is following the Kapustin-Witten equations case \cite{HuangLMP}. 
\section{Fundamental preliminaries }
We shall generally adhere to the now standard gauge-theory conventions and notation of Donaldson and Kronheimer \cite{Donaldson/Kronheimer} and Feehan \cite{Feehan2014.09}. Throughout our article, $\Om^{p}(X,\mathfrak{g}_{P})$ denote the smooth $p$-forms with values in $\mathfrak{g}_{P}$. Given a connection on $P$, we denote by $\na_{A}$ the corresponding covariant derivative on $\Om^{\ast}(X,\mathfrak{g}_{P})$ induced by $A$ and the Levi-Civita connection of $X$. Let $d_{A}$ denote the exterior derivative associated to $\na_{A}$. For $u\in L^{p}_{k}(X,\mathfrak{g}_{P})$, where $1\leq p<\infty$ and $k$ is an integer, we denote
\begin{equation}\nonumber
\|u\|_{L^{p}_{k}(X)}:=\big{(}\sum_{j=0}^{k}\int_{X}|\na^{j}_{A}u|^{p}dvol_{g}\big{)}^{1/p},
\end{equation}
where $\na^{j}_{A}:=\na_{A}\circ\ldots\circ\na_{A}$ (repeated $j$ times for $j\geq0$). For $p=\infty$, we denote
\begin{equation}\nonumber
\|u\|_{L^{\infty}_{k}(X)}:=\sum_{j=0}^{k}\operatorname{ess}\sup_{X}|\na^{j}_{A}u|.
\end{equation}
\subsection{Chern-Weil formula}
Given a connection $A$ on $P$, Chern-Weil theory provides representatives for the $first$ $Pontrjagin$ $class$ of $\mathfrak{g}_{P}$, namely 
$$p_{1}(P)=-\frac{1}{4\pi^{2}}\operatorname{tr}(F_{A}\wedge F_{A}),$$
and hence the $first$ $Pontrjagin$ $number$ 
$$p_{1}(P)[X]=-\frac{1}{4\pi^{2}}\int_{X}\operatorname{tr}(F_{A}\wedge F_{A}).$$
The Principal $G$-bundle $P$ are classified, see \cite[Propositions A.1 and A.2]{Taubes1982} by a cohomology class $\eta(P)\in H^{2}(X;\pi_{1}(G))$ and the  $first$ $Pontrjagin$ $number$ $p_{1}(P)[X]$

Assume in addition that $X$ is equipped with a Riemannian metric $g$, we recall some facts concerning the Killing form \cite{Feehan2014.09}. Every element $\xi$ of a Lie algebra $\mathfrak{g}$ over a field $\mathbb{K}$ defines an adjoint endomorphism, $\rm{ad}\xi\in End_{\mathbb{K}}\mathfrak{g}$, with the help of the Lie bracket via $(\rm{ad}\xi)(\zeta):=[\xi,\zeta]$, for all $\zeta\in\mathfrak{g}$. For a finite-dimensional Lie algebra $\mathfrak{g}$, its Killing form is the symmetric bilinear form, $B(\xi,\zeta):=\operatorname{tr}(\rm{ad}\xi\circ \rm{ad}\zeta)$, $\forall\xi,\zeta\in\mathfrak{g}$, with values in $\mathbb{K}$. The Killing form of a semisimple Lie algebra is negative definite. In particular, if $B_{\mathfrak{g}}$ is the Killing form on $\mathfrak{g}$, then it defines an inner product on $\mathfrak{g}$ via $(\cdot,\cdot)_{\mathfrak{g}}=-B_{\mathfrak{g}}(\cdot,\cdot)$ and thus a norm $|\cdot|_{\mathfrak{g}}$ on $\mathfrak{g}$.

The $4$-dimensional topological term for $-4\pi^{2}p_{1}(P)[X]$ extends algebraically to the complex function on $-4\pi^{2}p_{1}(P_{\mathbb{C}})[X]$:
\begin{equation*}
\begin{split}
\int_{X}\operatorname{tr}(\mathcal{F}_{\mathcal{A}}\wedge\mathcal{F}_{\mathcal{A}}):&=\int_{X}\operatorname{tr}\big{(}(F_{A}-\phi\wedge\phi)\wedge(F_{A}-\phi\wedge\phi)\\
&-(d_{A}\phi\wedge d_{A}\phi)+2\sqrt{-1}(F_{A}-\phi\wedge\phi)\wedge d_{A}\phi\big{)}.\\
\end{split}
\end{equation*}
The following proposition which proved by Gagliardo-Uhlenbeck \cite{Gagliardo/Uhlenbeck:2012}.
\begin{proposition}
	If $X$ is a closed manifold, then
	\begin{equation}\label{E11}
 \int_{X}\operatorname{tr}(\mathcal{F}_{\mathcal{A}}\wedge\mathcal{F}_{\mathcal{A}})=\int_{X}\operatorname{tr}(F_{A}\wedge F_{A}).
	\end{equation}
\end{proposition}
\begin{proof}
Noting that $\operatorname{tr}(\phi\wedge\phi\wedge\phi\wedge\phi)=0$. Then, it's easy to see that
\begin{equation*}
\begin{split}
&\quad\int_{X}\operatorname{tr}\big{(}(F_{A}-\phi\wedge\phi)\wedge(F_{A}-\phi\wedge\phi)\big{)}\\
&=\int_{X}\operatorname{tr}(F_{A}\wedge F_{A})-2\int_{X}\operatorname{tr}(F_{A}\wedge(\phi\wedge\phi))+\int_{X}\operatorname{tr}(\phi\wedge\phi\wedge\phi\wedge\phi)\\
&=\int_{X}\operatorname{tr}(F_{A}\wedge F_{A})-2\int_{X}\operatorname{tr}(F_{A}\wedge(\phi\wedge\phi)).\\
\end{split}
\end{equation*}
We also observe that
\begin{equation*}
\begin{split}
\int_{X}\operatorname{tr}(d_{A}\phi\wedge d_{A}\phi)&=\int_{X}\operatorname{tr} d_{A}(\phi\wedge d_{A}\phi)+\int_{X}\operatorname{tr}(\phi\wedge[F_{A},\phi])\\
&=2\int_{X}\operatorname{tr}(F_{A}\wedge(\phi\wedge\phi)),\\
\end{split}
\end{equation*}
and 
\begin{equation*}
\begin{split}
\int_{X}\operatorname{tr}(F_{A}-\phi\wedge\phi)\wedge d_{A}\phi&=\int_{X}\operatorname{tr}d_{A}(F_{A}\wedge\phi)-\frac{1}{3}\int_{X}\operatorname{tr}d_{A}(\phi\wedge\phi\wedge\phi)\\
&=\int_{X}d\operatorname{tr}(F_{A}\wedge\phi)-\frac{1}{3}\int_{X}d\operatorname{tr}(\phi\wedge\phi\wedge\phi)=0.\\
\end{split}
\end{equation*}
Combining the preceding identities gives (\ref{E11}).
\end{proof}
On a four-dimensional manifold, a form $u$ with value in $\Om^{2}(\mathfrak{g}_{P}^\mathbb{C})$, we can decompose $u$ to  $u^{+}:=\frac{1}{2}(u+\ast\bar{u})$ and $u^{-}:=\frac{1}{2}(u-\ast\bar{u})$, so it follows that
$\ast u^{+}=\bar{u}^{+}$, $\ast u^{-}=-\bar{u}^{-}$ and
$Re(u^{+}\wedge u^{-})=0$, then
$$-|u|^{2}dvol_{g}=\operatorname{tr}(u\wedge u)-2|u^{-}|^{2}dvol_{g}.$$
Hence we take $u=\mathcal{F}_{\mathcal{A}}$ to obtain
\begin{equation*}
\begin{split}
\|\mathcal{F}_{\mathcal{A}}\|^{2}_{L^{2}(X)}&=\|\mathcal{F}^{+}_{\mathcal{A}}\|^{2}_{L^{2}(X)}+\|\mathcal{F}^{-}_{\mathcal{A}}\|^{2}_{L^{2}(X)}\\
&=2\|\mathcal{F}^{-}_{\mathcal{A}}\|^{2}_{L^{2}(X)}+4\pi^{2}p_{1}(P)[X]\\
&=2\|\mathcal{F}^{+}_{\mathcal{A}}\|^{2}_{L^{2}(X)}-4\pi^{2}p_{1}(P)[X].\\
\end{split}
\end{equation*}
\subsection{Identities for the solutions}
In this section, we recall some basic identities that obeying the solutions to complex Yang-Mills connections. A nice discussion of these identities can be found in \cite{Gagliardo/Uhlenbeck:2012}. At first, we recall some Weitzenb\"{o}ck formulas \cite{Bourguignon/Lawson} and \cite[Equs.(6.25) and (6.26)]{Freed/Uhlenbeck}.
\begin{theorem}
	On four-dimensional manifold, we then have 
	\begin{subequations}
		\begin{numcases}{}
		2d_{A}^{-,\ast}d_{A}^{-}+d_{A}d^{\ast}_{A}=\na^{\ast}_{A}\na_{A}+Ric(\cdot)+2\ast[F^{+}_{A},\cdot]\ on\ \Om^{1}(X,\mathfrak{g}_{P})\\
		2d_{A}^{+}d_{A}^{+,\ast}=\na^{\ast}_{A}\na_{A}+(\frac{1}{3}R-2\w^{+})(\cdot)+\{F^{+}_{A},\cdot\} on\ \Om^{2}(X,\mathfrak{g}_{P}).
		\end{numcases}
	\end{subequations}
\end{theorem}
\begin{proposition}
	If the pair $(A,\phi)$ is a solution of the complex Yang-Mills equations over a four-dimensional manifold $X$, then the pair $(A,\phi)$ also satisfies the following equations
	\begin{subequations}\label{E3}
		\begin{numcases}{}
	d_{A}^{\ast}(F_{A}-\phi\wedge\phi)^{+}-\ast[\phi,(d_{A}\phi)^{-}]=0,\\
	d_{A}^{\ast}(d_{A}\phi)^{-}-\ast[\phi,(F_{A}-\phi\wedge\phi)^{+}]=0.
	\end{numcases}
\end{subequations}
\end{proposition}
\begin{proof}
	We  decompose the curvature $\mathcal{F}_{\mathcal{A}}$ of connection $\mathcal{A}$ to
	$$\mathcal{F}^{+}_{\mathcal{A}}:=\frac{1}{2}(\mathcal{F}_{\mathcal{A}}+\ast\overline{\mathcal{F}_{\mathcal{A}}})=\frac{1}{2}(F_{A}-\phi\wedge\phi)^{+}+\sqrt{-1}(d_{A}\phi)^{-}$$
	and
	$$\mathcal{F}^{-}_{\mathcal{A}}:=\frac{1}{2}(\mathcal{F}_{\mathcal{A}}-\ast\overline{\mathcal{F}_{\mathcal{A}}})=\frac{1}{2}(F_{A}-\phi\wedge\phi)^{-}+\sqrt{-1}(d_{A}\phi)^{+}$$ 
	Since the $\mathcal{F}_{\mathcal{A}}$ satisfies complex Yang-Mills equations (\ref{E9}) and $Bianchi$  $identity$,  we have
	$$d_{\mathcal{A}}\mathcal{F}^{+}_{\mathcal{A}}=d_{\mathcal{A}}\mathcal{F}^{-}_{\mathcal{A}}=0.$$
	We  complete the proof of this Proposition.
\end{proof}
\begin{proposition}(\cite[Theorem 4.3]{Gagliardo/Uhlenbeck:2012} )
	If the pair $(A,\phi)$ is a solution of the complex Yang-Mills equations, then
	$$\na^{\ast}_{A}\na_{A}\phi+Ric\circ\phi+\ast[\ast(\phi\wedge\phi),\phi]=0.$$
	If $X$ is a closed manifold,
	\begin{equation}\label{E8}
	\|\na_{A}\phi\|^{2}_{L^{2}(X)}+\langle Ric\circ\phi,\phi\rangle_{L^{2}(X)}+2\|\phi\wedge\phi\|^{2}_{L^{2}(X)}=0.
	\end{equation}
\end{proposition}
We have an Energy identity of the pair $(A,\phi)$ of a solution of complex Yang-Mills equations, one also can see \cite{HuangLMP}.
\begin{proposition}(Energy Identity)\label{P4}
	If the pair $(A,\phi)$ is a solution of the complex Yang-Mills equations, then
	$$YM_{\mathbb{C}}(A,\phi)=\|F_{A}\|^{2}_{L^{2}(X)}-\|\phi\wedge\phi\|^{2}_{L^{2}(X)}.$$
	In particular, if $X$ is a four-dimensional manifold,
	\begin{equation*}
	\begin{split}
	YM_{\mathbb{C}}(A,\phi)&=2\|\mathcal{F}^{+}_{\mathcal{A}}\|^{2}_{L^{2}(X)}-4\pi^{2}p_{1}(P)[X]\\
	&=2\|F_{A}^{+}\|^{2}_{L^{2}(X)}-2\|(\phi\wedge\phi)^{+}\|^{2}_{L^{2}(X)}-4\pi^{2}p_{1}(P)[X].\\
	\end{split}
	\end{equation*}
\end{proposition}
As one application of the maximum principle, Gagliardo-Uhlenbeck obtain an a $L^{\infty}$ priori estimate for the extra fields \cite[Corollary 4.6]{Gagliardo/Uhlenbeck:2012}.
\begin{theorem}
	If the pair $(A,\phi)$ is a smooth solution of complex Yang-Mills equations, then
	$$\|\phi\|_{L^{\infty}(X)}\leq C\|\phi\|_{L^{2}(X)},$$
	where $C=C(g)$ is a positive constant only dependent on $g$.
\end{theorem}
\section{A priori estimates of Complex Yang-Mills equations}
\subsection{The least eigenvalue of $d_{A}^{+}d_{A}^{+,\ast}$ when $F_{A}^{+}$ is $L^{2}$ small}
Consider the open neighborhood in $\mathcal{B}(P,g):=\mathcal{A}_{P}/\mathcal{G}_{P}$ which is the affine space, $\mathcal{A}_{P}$, of connection on $P$ moduli the gauge transformations, $\mathcal{G}_{P}$, of the principal $G$-bundle, of the finite-dimensional subvariety 
$$\mathcal{M}(P,g):=\{[A]\in\mathcal{B}(P,g):F_{A}^{+}=0\}$$
defined by 
$$\mathcal{B}_{\varepsilon}(P,g):=\{[A]\in\mathcal{B}(P,g):\|F_{A}^{+}\|_{L^{2}(X)}<\varepsilon\}.$$
The operator $d_{A}^{+}d_{A}^{\ast,+}$ is an elliptic self-dual operator on the space of square integrable sections of $\Om^{2,+}(X,\mathfrak{g}_{P})$. It is a standard result that the spectrum of $d_{A}^{+}d_{A}^{\ast,+}$ is discrete, and the lowest eigenvalue is nonnegative \cite[Page 146]{Taubes1982}. First recall the Definition of the least eigenvalue of $d_{A}^{+}d_{A}^{+,\ast}$,  \cite[Definition 3.1]{Taubes1982} ,
\begin{definition}\label{D2}
	For $A\in \mathcal{A}_{P}$, define
	\begin{equation}\label{E1}
	\mu_{g}(A):=\inf_{v\in\Om^{+}(\mathfrak{g}_{P})\backslash\{0\}}\frac{\|d^{+,\ast}_{A}v\|^{2}}{\|v\|^{2}}.
	\end{equation}
	is the lowest eigenvalue of $d_{A}^{+}d_{A}^{+,\ast}$.
\end{definition}
When the metric $g$ is positive in the sense of \cite[Definition 3.1]{Feehan2014.12} , then  $\mu_{g}(A)$ has a positive lower bound with respect to $[A]\in\mathcal{B}_{\varepsilon}(P,g)$. In \cite[Section 3]{Feehan2014.12} , the author proved $\mu_{g}(A)$ has a positive lower bound with respect to $[A]\in\mathcal{B}_{\varepsilon}(P,g)$ under the Riemannian metric $g$ is $good$.
\begin{theorem}\label{T4}(\cite[Theorem 3.8]{Feehan2014.12} )
	Let $X$ be a closed, oriented, four-dimensional, smooth Riemannian manifold with smooth Riemannian metric $g$, $P$ be a $G$-bundle over $X$ with $G$ being a compact Lie group. Assume that $g$ is $good$ in the sense of Definition \ref{D1}. Then there are positive constants $\mu=\mu(g,P)$ and $\varepsilon=\varepsilon(g,P)$ with following significance. If $A$ is a smooth connection on $P$ such that 
	\begin{equation}\nonumber
	\|F^{+}_{A}\|_{L^{2}(X)}\leq\varepsilon
	\end{equation}
	and $\mu_{g}(A)$ is as in (\ref{E1}),\ then
	$$\mu_{g}(A)\geq\mu.$$
\end{theorem}
\subsection{A solution of complex Yang-Mills equations when $F_{A}^{+}$ is $L^{2}$ small}
We construct the following useful a priori estimate. One also can see \cite[Lemma 5.2]{Taubes1982}  or \cite{Taubes1984}. 
\begin{lemma}\label{L2}
	Let $X$ be a closed, oriented, four-dimensional, smooth Riemannian manifold with smooth Riemannan metric $g$. Then there are positive constants, $C=C(g)$ and $\varepsilon=\varepsilon(g)$ with the following significance. If $G$ is a compact Lie group, $A$ is a smooth connection on a principal $G$-bundle $P$ over $X$ with 
	\begin{equation}\label{E4}
	\|F_{A}^{+}\|_{L^{2}(X)}\leq\varepsilon
	\end{equation}
	and $\nu\in\Om^{+}(\mathfrak{g}_{P})$, then
	\begin{equation}\label{E2}
	\|\nu\|^{2}_{L^{4}(X)}\leq C(\|d_{A}^{+,\ast}\nu\|^{2}_{L^{2}(X)}+\|\nu\|^{2}_{L^{2}(X)})
	\end{equation}
\end{lemma}
\begin{proof}
For any	 $\nu\in\Om^{+}(\mathfrak{g}_{P})$, by the Weitzenb\"{o}ck formula Equation (2.2b), we have
	$$\|\na_{A}\nu\|^{2}_{L^{2}}\leq 2\|d_{A}^{+,\ast}\nu\|^{2}_{L^{2}}+|\langle\frac{1}{3}R-2\w^{+})\circ\nu,\nu\rangle_{L^{2}}|+|\langle\{F_{A}^{+},\nu\},\nu\rangle_{L^{2}}|.$$
	We observe that
	$$|\langle\frac{1}{3}R-2\w^{+})\circ\nu,\nu\rangle_{L^{2}}|\leq c\|\nu\|^{2}_{L^{2}}$$
	and
	$$
	|\langle\{F_{A}^{+},\nu\},\nu\rangle_{L^{2}}|\leq c\|F^{+}_{A}\|_{L^{2}}\|\nu\|^{2}_{L^{4}}
	$$
	for some $c=c(g)$. Combining the preceding inequalities yields
	\begin{equation}\nonumber
	\begin{split}
	\|\nu\|^{2}_{L^{4}}&\leq c(\|\nu\|^{2}_{L^{2}}+\|\na_{A}\nu\|^{2}_{L^{2}})\\
	&\leq c(\|d_{A}^{+,\ast}\nu\|^{2}_{L^{2}}+\|\nu\|^{2}_{L^{2}}+\|F^{+}_{A}\|_{L^{2}(X)}\|\nu\|^{2}_{L^{4}}),\\
	\end{split}
	\end{equation}
	for some $c=c(g)$. Provided $c\|F^{+}_{A}\|_{L^{2}}
	\leq1/2$, rearrangement gives (\ref{E2}).
\end{proof}
Now, from Lemma \ref{L2} and the definition of $\mu_{g}(A)$, we have the following {a priori} estimate on the case of the Riemannian metric $g$ which has $\mu_{g}(A)>0$. One also can see \cite[Corollary 4.2]{Feehan2014.12} .
\begin{corollary}\label{C3}
Assume the hypothesis of Lemma \ref{L2}. Then there are positive constants, $C=C(g,G)$ and $\varepsilon=\varepsilon(g,P)$ with the following significance. If $A$ is a smooth connection  on $P$ obeying (\ref{E4}) and $\mu_{g}(A)>0$. If $\nu\in\Om^{+}(\mathfrak{g}_{P})$, then
\begin{equation}\label{E5}
\|\nu\|_{L^{4}(X)}\leq C(1+1/\sqrt{\mu_{g}(A)})\|d^{+,\ast}_{A}\nu\|_{L^{2}(X)}).
\end{equation}
\end{corollary}
We have the following useful a priori estimate on the solutions of complex Yang-Mills equations.
\begin{proposition}\label{P3}
Let $G$ be a compact, simple Lie group and P be a principal $G$-bundle over a closed, connected, four-dimensional, oriented, smooth manifold $X$ with Riemannian metric $g$ that is $good$ in the sense of Definition \ref{D1}. Then there are positive constants, $C=C(g,G)$ and $\varepsilon=\varepsilon(g,P)$ with the following significance. If the pair $(A,\phi)$ is a smooth solution of complex Yang-Mills on $P$ such that 
$$\|F_{A}^{+}\|_{L^{2}(X)}\leq\varepsilon$$
then
$$\|F_{A}^{+}\|_{L^{4}(X)}\leq C\|\phi\|^{2}_{L^{2}(X)}.$$
\end{proposition}
\begin{proof}
	Since $(A,\phi)$ satisfies complex Yang-Mills equations (\ref{E3}), we have             
	$$\|\na_{A}\phi\|^{2}_{L^{2}}=-\langle Ric\circ\phi,\phi\rangle_{L^{2}}-2\|\phi\wedge\phi\|^{2}_{L^{2}}\leq c\|\phi\|^{2}_{L^{2}}$$
	for some $c=c(g)$, and
	$$\|d_{A}^{+,\ast}F_{A}^{+}\|_{L^{2}}\leq\|d_{A}^{+,\ast}(\phi\wedge\phi)^{+}\|_{L^{2}}+\|[\phi,(d_{A}\phi)^{-}]\|_{L^{2}}.$$
	Nothing that $|d_{A}^{\ast}(\phi\wedge\phi)^{+}|\leq c|\na_{A}\phi|\cdot|\phi|$ for some constant $c=c(G)$. We then have
	\begin{equation}\nonumber
	\|d_{A}^{+,\ast}(\phi\wedge\phi)^{+}\|_{L^{2}}\leq c\|\na_{A}\phi\|_{L^{2}}\|\phi\|_{L^{\infty}}\leq c\|\na_{A}\phi\|_{L^{2}}\|\phi\|_{L^{2}}
	\end{equation}
	and
	\begin{equation}\nonumber
	|[\phi,(d_{A}\phi)^{-}]\|_{L^{2}}\leq c\|(d_{A}\phi)^{-}\|_{L^{2}}\|\phi\|_{L^{\infty}}\leq c\|\na_{A}\phi\|_{L^{2}}\|\phi\|_{L^{2}},
	\end{equation}
	for some $c=c(g)$. Combining the preceding inequalities yields
	\begin{equation}\label{E12}
	\|d_{A}^{+,\ast}F_{A}^{+}\|_{L^{2}}\leq c\|\phi\|^{2}_{L^{2}}
	\end{equation}
	For a smooth connection $A$ on $P$ with $\|F_{A}^{+}\|_{L^{2}}\leq\varepsilon$, where $\varepsilon=\varepsilon(g,P)$ is as in the hypothesis of Corollary \ref{C3}, we can apply the a priori estimate (\ref{E4}) to $\nu=F^{+}_{A}$ to obtain
	\begin{equation}\label{E13}
	\|F_{A}^{+}\|_{L^{4}}\leq c(1+1/\sqrt{\mu})\|d_{A}^{+,\ast}F_{A}^{+}\|_{L^{2}}),
	\end{equation}
	where $c=c(g)$ is as in Corollary \ref{C3} and $\mu=\mu(g,P)$ is the uniform positive lower {bound} for $\mu_{g}(A)$ provided by Theorem \ref{T4}. Combining the  inequalities (\ref{E12})--(\ref{E13}) yields
	$$
	\|F_{A}^{+}\|_{L^{4}}\leq c\|\phi\|^{2}_{L^{2}},$$
	for some $c=c(g,P)$. 
\end{proof}
{\begin{proof}[\textbf{Proof of Corollary \ref{C2}.}] Since pure Yang-Mills equation is a special case of complex Yang-Mills equations, i.e., $\phi=0$, the conclusion follows from Proposition \ref{P3}.
\end{proof}}
If $\G(\cdot,\cdot)$ denotes the Green kernel of the Laplace operator, $d^{\ast}d$, on $\Om^{2}(X)$, we define
\begin{equation}\nonumber
\|v\|_{L^{\sharp,2}(X)}:=\sup_{x\in X}\int_{X}\G(x,y)|v|(y)dvol_{g}(y)+\|v\|_{L^{2}(X)},\ \forall v\in\Om^{2}(\mathfrak{g}_{P}).
\end{equation} We recall that $\G(x,y)$ has a singularity comparable with $\operatorname{dist}_{g}(x,y)^{-2}$, when $x,y\in X$ are close. The norm $\|v\|_{L^{2}(X)}$ is conformally invariant and 
$$\|\nu\|_{L^{\sharp}(X)}:=\sup_{x\in X}\int_{X}\G(x,y)|v|(y)dvol_{g}(y)$$
is scale invariant \cite{Chavel}.  Indeed, one sees this by noting that if $g\mapsto\tilde{g}=\la^{-2}g$, then $\operatorname{dist}_{\tilde{g}}(x,y)=\la^{-1}\operatorname{dist}_{g}(x,y)$ and $dvol_{\tilde{g}}=\la^{-1}dvol_{g}$, while for $v\in\Om^{2}(X,\mathfrak{g}_{P})$, we have $|v|_{\tilde{g}}=\la^{2}|v|_{g}$. One can see that \cite[Lemma 4.1]{Feehan}
$$\|\nu\|_{L^{\sharp}(X)}\leq c_{p}\|\nu\|_{L^{p}(X)},\ \forall  \ p>2,$$
where $c_{p}$ depends at most on $p$ and the Riemannian metric, $g$, on $X$.

Let $A$ be a connection on principal $G$-bundle $P$ over $X$. The ASD connection for a second connection $A+a$, where $a\in\Om^{1}(X,\mathfrak{g}_{P})$ is a bundle-value $1$-form, can be written
\begin{equation}\label{EA0}
d^{+}_{A}a+(a\wedge a)^{+}=-F_{A}^{+}.
\end{equation} 
Following Taubes we seek a solution of the equation (\ref{EA0}) in the form 
$$a=d_{A}^{+,\ast}u$$
where $u\in\Om^{2,+}(X,\mathfrak{g}_{P})$ is a bundle-value self-dual $2$-form. Then (\ref{EA0}) becomes to a second order equation:
\begin{equation}\label{EA1}
d_{A}^{+}d_{A}^{+,\ast}u+(d_{A}^{+,\ast}u\wedge d_{A}^{+,\ast}u)^{+}=-F^{+}_{A}.
\end{equation}
If $X=S^{4}$ and $g$ is the stand metric on four-sphere, then $\frac{S}{3}-2w^{+}=2$. By Equation (2.2b), the Equation can be rewritten to 
$$\na^{\ast}_{A}\na_{A}u+u+\frac{1}{2}(d_{A}^{+,\ast}u\wedge d_{A}^{+,\ast}u)^{+}+\{F_{A}^{+},u\}=-\frac{1}{2}F_{A}^{+}.$$
Donaldson \cite{Donaldson1993} considered a general PDE
	$$\na_{A}^{\ast}\na_{A}u+u+b(\na_{A}u,\na_{A}u)+f(u)=\rho,$$
	where $b$ is symmetric, bilinear, bundle map. If $\|f\|_{L^{\sharp}}<1$ and $\|\rho\|_{L^{\sharp}}<\frac{(1-\|f\|_{L^{\sharp}})^{2}}{8}$, Donaldson \cite[Proposition 2]{Donaldson1993} proved that there is a unique solution $u$ to above equation with $\|u\|_{L^{\infty}}\leq \frac{8\|\rho\|_{L^{\sharp}}}{1-\|f\|_{L^{\sharp}}}$. Feehan-Leness \cite{Feehan/Leness} extended the ideas of Donaldson \cite{Donaldson1993} and Taubes \cite{Taubes1984} to the general Riemannian four-manifold case. We give a sketch of the proof of Theorem \ref{T3} in Appendix. Our proof is similar to the way in  \cite[Proposition  7.6]{Feehan/Leness}.  
\begin{theorem}\label{T3}
	Let $G$ be a compact Lie group, $P$ a principal $G$-bundle over a compact, connected, four-dimensional manifold, $X$, with Riemannian metric, $g$, and $E_{0},\mu\in(0,\infty)$ constants. Then there are constants, $C_{0}=C_{0}(g,\mu)\in(0,\infty)$ and $\eta=\eta(g,\mu)\in(0,1]$, with the following significance. If $A$ is a $C^{\infty}$ connection on $P$ such that
	\begin{equation}\nonumber
	\begin{split}
	&\mu_{g}(A)\geq\mu,\\
	&\|F_{A}^{+}\|_{L^{\sharp,2}(X)}\leq\eta,\\
	&\|F_{A}\|_{L^{2}(X)}\leq E_{0},\\
	\end{split}
	\end{equation}
	then there is a unique solution $u\in\Om^{+}(\mathfrak{g}_{P})$ to equation (\ref{EA1}) such that 
	$$\|u\|_{L^{2}_{1}(X)}\leq C_{0}\|F^{+}_{A}\|_{L^{4/3}(X)}.$$
\end{theorem}
Thus, we have a useful 
\begin{proposition}\label{P1}
Let $G$ be a compact Lie group, $P$ a principal $G$-bundle over a compact, connected, four-dimensional manifold, $X$, with Riemannian metric, $g$, and $E\in[1,\infty)$. Assume that $g$ is $good$ in the sense of Definition \ref{D1}. Then there exist constants $c=c(g,P,G)\in(0,\infty)$ and $K=K(g,P,G)$ with following significance. Let $(A,\phi)$ be a smooth solution of complex Yang-Mills over $X$ satisfies
	\begin{equation}\nonumber
	\begin{split}
	&\|\phi\|_{L^{2}(X)}\leq E,\\
	&\|F_{A}^{+}\|_{L^{2}(X)}\leq \varepsilon.\\
	\end{split}
	\end{equation}
{If $c\varepsilon^{1/3}E^{4/3}\leq cK\leq\eta$, where $\eta$ is the constant as in the hypothesis of Theorem \ref{T3},} then there is a anti-self-dual connection, $A_{\infty}$ on $P$, of class $C^{\infty}$ such that
	$$\|A_{\infty}-A\|_{L^{2}(X)}\leq c\|F^{+}_{A}\|_{L^{2}(X)}.$$
\end{proposition}
\begin{proof}
	From the Proposition \ref{P3}, we have 
	$$\|F^{+}_{A}\|_{L^{4}}\leq c\|\phi\|_{L^{2}}^{2}\leq cE^{2}$$
	for some $c=c(g)$. Then interpolation inequality gives
	$$\|F_{A}^{+}\|_{L^{3}}\leq \|F_{A}^{+}\|^{1/3}_{L^{2}}\|F_{A}^{+}\|^{2/3}_{L^{4}}.$$
	Hence, we have
	$$
	\|F_{A}^{+}\|_{L^{\sharp,2}}\leq c\|F_{A}^{+}\|_{L^{3}}\leq c\varepsilon^{1/3}E^{4/3},
	$$
	for some $c=c(g)$. Then following  Theorem \ref{T3}, there exist a unique solution $u\in\Om^{+}(\mathfrak{g}_{P})$ such that $A_{\infty}:=A+d^{+,\ast}_{A}u$ is anti-self-dual with respect to metric $g$ and $u$ satisfies 
	$$\|d_{A}^{+,\ast}u\|_{L^{2}}\leq c\|u\|_{L^{2}_{1}}\leq c\|F_{A}^{+}\|_{L^{2}},$$
where $c=c(g,P,G)$ is a positive constant.
\end{proof}

\subsection{ Complete the proof of Theorem \ref{T1}}
At first, we would construct an estimate between $F_{A}^{+}$ and $(\phi\wedge\phi)^{+}$.
\begin{proposition}\label{P2}
	Let $G$ be a compact Lie group, $P$ a principal $G$-bundle over a closed, connected, four-dimensional manifold, $X$, with Riemannian metric, $g$, that  is  good in the sense of Definition \ref{D1}. Then there are constants $\varepsilon=\varepsilon(g,P)\in(0,1)$, $c=c(g,P)\in[1,\infty)$ with following significance. If $(A,\phi)$ is a smooth solution of complex Yang-Mills over $X$ such that
	$$
	\|F_{A}^{+}\|_{L^{2}(X)}\leq \varepsilon,
	$$
	then 
	$$\|F_{A}^{+}\|_{L^{2}(X)}\leq c(1+\|\phi\|^{2}_{L^{2}(X)})\|(\phi\wedge\phi)^{+}\|_{L^{2}(X)}.$$
\end{proposition}
\begin{proof}
	Since $(A,\phi)$ satisfies complex Yang-Mills equations, we have
	\begin{equation}\nonumber
	\begin{split}
	\|(d_{A}\phi)^{-}\|^{2}_{L^{2}}&=2\langle (F_{A}-\phi\wedge\phi)^{+},(\phi\wedge\phi)^{+}\rangle_{L^{2}}\\
	&\leq 2\|(F_{A}-\phi\wedge\phi)^{+}\|_{L^{2}}\|(\phi\wedge\phi)^{+}\|_{L^{2}}
	\end{split}
	\end{equation}
	and
	$$\|d_{A}^{+,\ast}(F_{A}-\phi\wedge\phi)^{+}\|_{L^{2}}=|[\phi,(d_{A}\phi)^{-}]\|_{L^{2}}.$$
	We observe that
	\begin{equation}\nonumber
	\begin{split}
	\|(F_{A}-\phi\wedge\phi)^{+}\|^{2}_{L^{2}}&\leq (1/\mu)\|d_{A}^{+,\ast}(F_{A}-\phi\wedge\phi)^{+}\|^{2}_{L^{2}}\\
	&\leq c\|(d_{A}\phi)^{-}\|^{2}_{L^{2}}\|\phi\|^{2}_{L^{2}},\\
	\end{split}
	\end{equation}
	for some $c=c(g)$. Combining the preceding inequalities yields
	\begin{equation}\nonumber
	%\begin{split}
	\|(F_{A}-\phi\wedge\phi)^{+}\|_{L^{2}}^{2}\leq c\|(F_{A}-\phi\wedge\phi)^{+}\|_{L^{2}}\|(\phi\wedge\phi)^{+}\|_{L^{2}}\|\phi\|^{2}_{L^{2}}
	%\end{split}
	\end{equation}
	for some $c=c(g)$. Hence we have
	$$\|(F_{A}-\phi\wedge\phi)^{+}\|_{L^{2}}\leq c\|(\phi\wedge\phi)^{+}\|_{L^{2}}\|\phi\|^{2}_{L^{2}}.$$
	Therefore, we have
	$$\|F_{A}^{+}\|_{L^{2}}\leq c(1+\|\phi\|^{2}_{L^{2}})\|(\phi\wedge\phi)^{+}\|_{L^{2}}.$$
	We complete the proof of this proposition.
\end{proof}
%\begin{corollary}
%Assume the hypothesis of Proposition \ref{P2}. Let $E\in (0,\infty)$ be a positive constant. Then there are constants $ \varepsilon=\varepsilon(g,P,\varepsilon)$, $c=c(g,P)$ with following significance. If $(A,\phi)$ is a smooth solution of complex Yang-Mills over $X$ such that
%\begin{equation*}
%\begin{split} 
%&\|\phi\|_{L^{2}(X)}\leq E,\\
%&\|F_{A}^{+}\|_{L^{2}(X)}\leq \varepsilon,\\
%\end{split}
%\end{equation*}
%then 
%$$\|F_{A}^{+}\|_{L^{2}(X)}\leq c(1+E^{2})\|(\phi\wedge\phi)^{+}\|_{L^{2}(X)}.$$
%\end{corollary}
\begin{proof}[\textbf{Proof of Theorem \ref{T1}.}]The idea is following the case of Kapustin-Witten equations \cite{HuangLMP}. {Let $\varepsilon=\varepsilon(g,P)$ be a positive constant satisfying the conclusions of Theorem \ref{T3}.} Suppose that $(A,\phi)$ is a smooth solution of complex Yang-Mills equation with $0\neq\|F_{A}^{+}\|_{L^{2}}\leq\varepsilon$. {Let $E_{0}=\frac{K^{3/4}}{\varepsilon^{1/4}}$, where $K$ is the positive constant as in the hypothesis of Proposition \ref{P1}.} For $\phi$ with $\|\phi\|_{L^{2}}\leq E\leq E_{0}$,  then {$\|F^{+}_{A}\|_{L^{2,\sharp}}\leq cK\leq \eta$.} Then following Proposition \ref{P1}, there exists a anti-self-dual connection $A_{\infty}$ such that 
	$$\|A_{\infty}-A\|_{L^{2}}\leq c\|F_{A}^{+}\|_{L^{2}},$$
	for some $c=c(g)$. The {Weitzenb\"{o}ck} formula (2.2b) gives,
	\begin{equation}\nonumber
	(2d^{-,\ast}_{A_{\infty}}d^{-}_{A_{\infty}}+d_{A_{\infty}}d_{A_{\infty}}^{\ast})\phi=\na_{A_{\infty}}^{\ast}\na_{A_{\infty}}\phi+Ric\circ\phi.
	\end{equation}
	It provides an integral inequality
	\begin{equation}\nonumber
	\|\na_{A_{\infty}}\phi\|^{2}_{L^{2}}+\langle Ric\circ\phi,\phi\rangle_{L^{2}}\geq0.
	\end{equation}
	Since $(A,\phi)$ satisfies complex Yang-Mills equations, we also have an integral identity 
	\begin{equation}\nonumber
	\|\na_{A}\phi\|^{2}_{L^{2}}+\langle Ric\circ\phi,\phi\rangle_{L^{2}(X)}+4\|(\phi\wedge\phi)^{+}\|^{2}_{L^{2}}=0.
	\end{equation}
	Combining the preceding inequalities yields
	\begin{equation}\nonumber
	\begin{split}
	4\|(\phi\wedge\phi)^{+}\|^{2}_{L^{2}}&=-\|\na_{A}\phi\|^{2}_{L^{2}}-\langle Ric\circ\phi,\phi\rangle_{L^{2}}\\
	&\leq-\|\na_{A_{\infty}}\phi\|^{2}_{L^{2}}-\langle Ric\circ\phi,\phi\rangle+\|\na_{A}\phi-\na_{A_{\infty}}\phi\|_{L^{2}}^{2}\\
	&\leq\|[A-A_{\infty},\phi]\|_{L^{2}}^{2}\\
	&\leq c\|A-A_{\infty}\|_{L^{2}}^{2}\|\phi\|^{2}_{L^{\infty}}\\
	&\leq c\|F_{A}^{+}\|_{L^{2}}^{2}\|\phi\|^{2}_{L^{2}}
	\end{split}
	\end{equation}
	for some $c=c(g)$. Following Proposition \ref{P2}, we have
	$$\|(\phi\wedge\phi)^{+}\|^{2}_{L^{2}}\leq c\|(\phi\wedge\phi)^{+}\|^{2}_{L^{2}}(1+\|\phi\|^{2}_{L^{2}})^{2}\|\phi\|^{2}_{L^{2}}\leq \frac{c(1+E^{2})^{2}E^{2}}{4}\|(\phi\wedge\phi)^{+}\|^{2}_{L^{2}}.$$
{We can choose $E$ small enough so that the above is contradiction with $F_{A}^{+}\neq0$.} We complete the proof of Theorem \ref{T1}.
\end{proof} 
\begin{proof}
[\textbf{Proof of Corollary \ref{C1}.}] If the pair $(A,\phi)$ is a smooth solution of Kapustin-Witten equations, then
\begin{equation*}
\|F^{+}_{A}\|_{L^{2}}=\|(\phi\wedge\phi)^{+}\|_{L^{2}}\leq 2\|\phi\|^{2}_{L^{2}}.
\end{equation*}
{We can provide $\|\phi\|^{2}_{L^{2}}$ small enough to ensure that $	\|F^{+}_{A}\|_{L^{2}}\leq\varepsilon$ and $c\|F^{+}_{A}\|_{L^{2}}^{1/3}\|\phi\|_{L^{2}}^{3/4}\leq c\|\phi\|_{L^{2}}^{2}\leq\eta$, where the positive constants $\varepsilon,\eta$ are as in the hypotheses of Proposition \ref{P1}.} Hence following the way in Theorem \ref{T1}, we can prove that the connection $A$ is anti-self-dual with respect to the metric $g$. 
\end{proof}

\section*{Appendix: Approximate ASD connection}
At the request of the anonymous referee, we give a sketch of the proof of Theorem \ref{T3}.  Suppose that $\mu_{g}(A)>0$ for $[A]\in\mathcal{B}_{\varepsilon}(P,g)$, i.e., the Laplacian 
$$d_{A}^{+}d_{A}^{+,\ast}: L^{2}_{k+1}(\Om^{2,+}\otimes\mathfrak{g}_{P})\rightarrow  L^{2}_{k-1}(\Om^{2,+}\otimes\mathfrak{g}_{P})$$	
is inverse. For the purposes of applying Banach-space fixed theory to solve the non-linear equation (\ref{EA1}), it is convenient to write
$$u=G_{A}\xi,\ for\ some\ \xi\in L^{2}_{k-1}(\Om^{2,+}\otimes\mathfrak{g}_{P}),$$
where $G_{A}:L^{2}_{k-1}(\Om^{2,+}\otimes\mathfrak{g}_{P})\rightarrow L^{2}_{k+1}(\Om^{2,+}\otimes\mathfrak{g}_{P})$ is the Green's operator for the Laplacian $d_{A}^{+}d_{A}^{+,\ast}$, i.e., $d_{A}^{+}d_{A}^{+,\ast}G_{A}=Id$ for any $\xi\in L^{2}_{k-1}(\Om^{2,+}\otimes\mathfrak{g}_{P})$. Therefore, it suffices to solve for $\xi$ such that
\begin{equation}\label{EA5}
\xi+(d_{A}^{+,\ast}G_{A}\xi\wedge d_{A}^{+,\ast}G_{A}\xi)^{+}=-F_{A}^{+}.
\end{equation}

\begin{lemma}\label{AL2}(\cite[Corollary 5.10]{Feehan})
Let $X$ be a closed, oriented, smooth four-manifold with Riemannian
metric $g$. Then there are positive constants $c=c(g)$ and  $\varepsilon=\varepsilon(c)$ with the following significance. Let $A$ be an smooth connection on $G$-bundle $P$ over $X$ such that $\|F_{A}^{+}\|_{L^{\sharp,2}(X)}<\varepsilon$. Then the following estimate holds for any $v\in L^{\sharp,2}(\Om^{2,+}\otimes\mathfrak{g}_{P})$:
$$\|d_{A}^{+,\ast}v\|_{L^{2}_{1}(X)}\leq c(1+\|F_{A}\|_{L^{2}(X)})(\|d_{A}^{+}d_{A}^{+,\ast}v\|_{L^{\sharp,2}(X)})+\|v\|_{L^{2}(X)}.$$
\end{lemma}
Following the idea in \cite[Proposition 7.3]{Feehan}, we then have
\begin{proposition}\label{AP1}
Let $X$ be a closed, oriented, smooth four-manifold with Riemannian
metric $g$. Then is a positive constant $\varepsilon$ with the following significance. Let $A$ be an smooth connection on $G$-bundle $P$ over $X$ such that $\|F_{A}^{+}\|_{L^{\sharp,2}(X)}<\varepsilon$. Let $\mu$ be a positive constant. If $\mu_{g}(A)\geq\mu$, then there is a positive constant $C=C(\mu,g,P)$ such that for all $\xi\in L^{\sharp,2}(\Om^{2,+}\otimes\mathfrak{g}_{P})$,
$$\|d^{+,\ast}_{A}G_{A}\xi\|_{L^{2}(X)}\leq C\|\xi\|_{L^{\frac{4}{3}}(X)},$$
$$	\|d_{A}^{+,\ast}G_{A}\xi\|_{L^{2}(X)}\leq C\|\xi\|_{L^{\sharp,2}(X)}.$$
\end{proposition}
\begin{proof}
We  set $v=G_{A}\xi$, i.e, $d_{A}^{+}d_{A}^{+,\ast}v=\xi$. We first consider the $L^{2}$-estimate for $a:=d_{A}^{+,\ast}G_{A}\xi$. We integrate by parts to get
$$\|a\|^{2}_{L^{2}}=\|d_{A}^{+,\ast}v\|_{L^{2}}^{2}=\langle d_{A}^{+}d_{A}^{+,\ast}v,v\rangle_{L^{2}}=\langle \xi,v\rangle_{L^{2}}.$$
By Corollary \ref{C3} and Sobolev inequality, we then have
\begin{equation}\nonumber
\|a\|^{2}_{L^{2}}\leq\|\xi\|_{L^{\frac{4}{3}}}\|v\|_{L^{4}}\leq c(1+\mu^{-\frac{1}{2}})\|\xi\|_{L^{\frac{4}{3}}}\|d_{A}^{+,\ast}v\|_{L^{2}}.
\end{equation}
Therefore,
$$\|d_{A}^{+,\ast}v\|_{L^{2}}\leq c\|\xi\|_{L^{\frac{4}{3}}},$$
where $c=c(g,\mu)$ is a positive constant. 

We next derive the $L^{2}_{1}$ estimate for $a$. Since $a=d_{A}^{+,\ast}v$, then
\begin{equation}\nonumber
\begin{split}
\|a\|_{L^{2}}=\|d_{A}^{+,\ast}v\|_{L^{2}}&=\langle d_{A}^{+}d_{A}^{+,\ast}v,v\rangle_{L^{2}}\leq 2\|d_{A}^{+}d_{A}^{+,\ast}v\|_{L^{2}}+2\|v\|_{L^{2}}\\
&\leq 2(1+\mu^{-1})\|d_{A}^{+}d_{A}^{+,\ast}v\|_{L^{2}}.\\
\end{split}
\end{equation}
By Lemma \ref{AL2}, we have
\begin{equation}\nonumber
\begin{split}
\|a\|_{L^{2}_{1}}&\leq C(1+\|F_{A}\|_{L^{2}})(\|d_{A}^{+}d_{A}^{+,\ast}v\|_{L^{\sharp,2}}+\|v\|_{L^{2}})\\
&\leq C(1+\|F_{A}\|_{L^{2}})(\|d_{A}^{+}d_{A}^{+,\ast}v\|_{L^{\sharp,2}}+\mu^{-1}\|d_{A}^{+}d_{A}^{+,\ast}v\|_{L^{2}})\\
&\leq C(1+\|F_{A}\|_{L^{2}})(1+\mu^{-1})\|\xi\|_{L^{\sharp,2}}.\\
\end{split}
\end{equation}
\end{proof} 
\begin{lemma}(\cite[Lemma 7.2.23]{Donaldson/Kronheimer} and \cite[Lemma 7.5]{Feehan/Leness})\label{AL1}
Let $q:\mathcal{B}\rightarrow\mathcal{B}$ be a continuous map on a Banach space $\mathcal{B}$ with $q(0)=0$ and 
$$\|q(x_{1})-q(x_{2})\|\leq K(\|x_{1}\|+\|x_{2}\|)\|x_{1}-x_{2}\|$$
for some positive constant $K$ and all $x_{1},x_{2}$ in $\mathcal{B}$. Then for each $y$ in $\mathcal{B}$ with $\|y\|<\frac{1}{10K}$ there is a unique $x$ in $\mathcal{B}$ such that $\|x\|\leq\frac{1}{5K}$ and $x+q(x)=y$.
\end{lemma}
\begin{proof}[\textbf{Proof of Theorem \ref{T3}}] We try to solve equation (\ref{EA1}) for the solutions of the form $u=G_{A}\xi$, where $\xi\in L^{\sharp,2}(\Om^{2,+}\otimes\mathfrak{g}_{P})$ and $G_{A}$ is the Green's operator for the Laplacian $d_{A}^{+}d_{A}^{+,\ast}$. We now seek solutions to equation (\ref{EA5}). We now apply Lemma \ref{AL1} to equation (\ref{EA5}) with $\mathcal{B}=L^{\sharp,2}(\Om^{2,+}\otimes\mathfrak{g}_{P})$, choosing $y=-F^{+}_{A}$ and $q(\xi)=(d_{A}^{+,\ast}G_{A,\mu}\xi\wedge d_{A}^{+,\ast}G_{A,\mu}\xi)^{+}$, so our goal is to solve
$$\xi+q(\xi)=-F_{A}^{+}\ on\ L^{\sharp,2}(\Om^{2,+}\otimes\mathfrak{g}_{P}).$$
For any $\xi_{1},\xi_{2}\in L^{\sharp,2}(\Om^{2,+}\otimes\mathfrak{g}_{P})$, the estimate in \cite[Proposition 7.3]{Feehan/Leness} yields
$$\|d_{A}^{+,\ast}G_{A}\xi_{i}\|_{L^{2}_{1}}\leq C\|\xi_{i}\|_{L^{\sharp,2}},$$
where $C=(C,\mu,\|F_{A}\|)$.

We define 
$$\|u\|_{L^{2\sharp}}:=\sup_{x\in X}\|\operatorname{dist}_{g}^{-1}(x,\cdot)|u|\|_{L^{2}}.$$
Then we have a continuous embedding $L_{1}^{2}\subset L^{2\sharp}$ \cite[Lemma 4.1]{Feehan} and $L^{2\sharp}\otimes L^{2\sharp}\rightarrow L^{\sharp}$ is continuous \cite[Lemma 4.3]{Feehan}. By the H\"{o}lder’s inequality  and the $L^{\sharp,2}$-family of embedding and multiplication results of Feehan, we obtain
\begin{equation}\nonumber
\begin{split}
\|q(\xi_{1})-q(\xi_{2})\|_{L^{\sharp,2}}&=\|q(\xi_{1})-q(\xi_{2})\|_{L^{2}}+\|q(\xi_{1})-q(\xi_{2})\|_{L^{\sharp}}\\
&\leq K_{1}(\|d_{A}^{+,\ast}G_{A}\xi_{1}\|_{L^{4}}+\|d_{A}^{+,\ast}G_{A}\xi_{i}\|_{L^{4}})(\|d_{A}^{+,\ast}G_{A}\xi_{1}-d_{A}^{+,\ast}G_{A}\xi_{i}\|_{L^{4}})\\
&+K_{1}(\|d_{A}^{+,\ast}G_{A}\xi_{1}\|_{L^{2\sharp}}+\|d_{A}^{+,\ast}G_{A}\xi_{i}\|_{L^{2\sharp}})(\|d_{A}^{+,\ast}G_{A}\xi_{1}-d_{A}^{+,\ast}G_{A}\xi_{i}\|_{L^{2\sharp}})\\
&\leq K_{2}(\|d_{A}^{+,\ast}G_{A}\xi_{1}\|_{L^{2}_{1}}+\|d_{A}^{+,\ast}G_{A}\xi_{i}\|_{L^{2}_{1}})(\|d_{A}^{+,\ast}G_{A}\xi_{1}-d_{A}^{+,\ast}G_{A}\xi_{i}\|_{L^{2}_{1}})\\
&\leq CK_{2}(\|\xi_{1}\|_{L^{\sharp,2}}+\|\xi_{2}\|_{L^{\sharp,2}})\|\xi_{1}-\xi_{2}\|_{L^{\sharp,2}},\\
\end{split}
\end{equation}
where $K_{2}$ is a universal constant. Thus, provided $\varepsilon\leq \frac{1}{10 CK_{2}}$, we have $\|F^{+}_{A}\|_{L^{\sharp,2}}\leq\frac{1}{10 CK_{2}}$ and Lemma \ref{AL1} implies that there is a unique solution $\xi\in L^{\sharp,2}(\Om^{2,+}\otimes\mathfrak{g}_{P})$ to equation (\ref{EA5}) such that $\|\xi\|_{L^{\sharp,2}}\leq\frac{1}{5CK_{2}}$.
 
By Proposition \ref{AP1} and H\"{o}lder inequality, we observe that
\begin{equation*}
\begin{split}
\|q(\zeta_{1})-q(\zeta_{2})\|_{L^{\frac{4}{3}}}&\leq \|d_{A}^{+,\ast}G_{A}\zeta_{1}-d_{A}^{+,\ast}G_{A}\zeta_{2}\|_{L^{2}}\|d_{A}^{+,\ast}G_{A}\zeta_{1}+d_{A}^{+,\ast}G_{A}\zeta_{1}\|_{L^{4}}\\
&\leq c\|\zeta_{1}-\zeta_{2}\|_{L^{\frac{4}{3}}}(\|\zeta_{1}\|_{L^{\sharp,2}}+\|\zeta_{2}\|_{L^{\sharp,2}}),\\
\end{split}
\end{equation*}
where $c=c(\mu,g,\|F_{A}\|)$.

We denote $\zeta_{k}=\xi_{k}-\xi_{k-1}$ and $\zeta_{1}=\xi_{1}$, then
$$\xi_{1}=-F^{+}_{A},\  \xi_{2}=q(\xi_{1})$$
and $$\xi_{k}=q(\sum_{i=1}^{k-1}\xi_{i})-q(\sum_{i=1}^{k-2}\xi_{i}),\ \forall\ k\geq 3.$$
It is easy to show that, under the assumption of $F^{+}_{A}$, the sequence $\zeta_{k}$ defined by
$$\zeta_{k}=q(\zeta_{k-1})-F^{+}_{A},$$
starting with $\xi_{1}=-F^{+}_{A}$, is Cauchy with respect to $L^{\sharp,2}$, and so converges to a limit $\zeta$ in the completion of $\Om^{2,+}\otimes\mathfrak{g}_{P}$ under $L^{\sharp,2}$.

There are positive constant $\varepsilon\in(0,1)$ and $C\in(1,\infty)$ with following significance. If the connection $A$ satisfies $$\|F^{+}_{A}\|_{L^{\sharp,2}}\leq\varepsilon,$$
then each $\xi_{k}$ exists and is $C^{\infty}$. Further for each $k\geq1$, we have
\begin{equation}\label{EA10}
\|\xi_{k}\|_{L^{\frac{4}{3}}}\leq C^{k-1}\|F^{+}_{A}\|_{L^{\frac{4}{3}}}\|F^{+}_{A}\|^{k-1}_{L^{\sharp,2}}.
\end{equation}
\begin{equation}\label{EA11}
\|\xi_     {k}\|_{L^{\sharp,2}}\leq C^{k-1}\|F^{+}_{A}\|^{k-1}_{L^{\sharp,2}}.
\end{equation}
The proof is by induction on the integer $k$. The induction begins with $k=1$, one can see $\xi_{1}=-F^{+}_{A}$. The induction proof if completed by demonstrating that if (\ref{EA10}) and (\ref{EA11}) are satisfied for $j<k$, then these also satisfied for $j=k$. Indeed, since
\begin{equation}\nonumber
\begin{split}
\|q(\sum_{i=1}^{k-1}\xi_{i})-q(\sum_{i=1}^{k-2}\xi_{i})\|_{L^{\sharp,2}}
&\leq c\|\sum_{i=1}^{k-1}\xi_{i}+\sum_{i=1}^{k-2}\xi_{i}\|_{L^{2,\sharp}}\|\xi_{k-1}\|_{L^{\sharp,2}},\\
&\leq 2c\sum\|\xi_{i}\|_{L^{\sharp,2}}\|\xi_{k-1}\|_{L^{\sharp,2}},\\
&\leq 2c\frac{\|F^{+}_{A}\|_{L^{\sharp,2}}}{1-C\|F^{+}_{A}\|_{L^{,\sharp,2}}}C^{k-2}\|F^{+}_{A}\|^{k-2}_{L^{\sharp,2}},\\
&\leq \frac{2c}{1-C\|F^{+}_{A}\|_{L^{,\sharp,2}}}C^{k-2}\|F^{+}_{A}\|^{k-1}_{L^{\sharp,2}}.\\
\end{split}
\end{equation}
Now, we provide the constants $\varepsilon$ sufficiently small and $C$ sufficiently large to ensures that
$$\frac{2c}{1-C\|F^{+}_{A}\|_{L^{\sharp,2}}}\leq C,$$  
hence  we complete the Equation (\ref{EA11}) is also satisfied for $j=k$. For the Equation (\ref{EA10}) case, we observe that
\begin{equation}\nonumber
\begin{split}
\|q(\sum_{i=1}^{k-1}\xi_{i})-q(\sum_{i=1}^{k-2}\xi_{i})\|_{L^{\frac{4}{3}}}
&\leq c\|\sum_{i=1}^{k-1}\xi_{i}+\sum_{i=1}^{k-2}\xi_{i}\|_{L^{2,\sharp}}\|\xi_{k-1}\|_{L^{\frac{4}{3}}},\\
&\leq 2c\sum\|\xi_{i}\|_{L^{\sharp,2}}\|\xi_{k-1}\|_{L^{\frac{4}{3}}},\\
&\leq 2c\frac{\|F^{+}_{A}\|_{L^{\sharp,2}}}{1-C\|F^{+}_{A}\|_{L^{,\sharp,2}}}C^{k-2}\|F^{+}_{A}\|_{L^{\frac{4}{3}}}\|F^{+}_{A}\|^{k-2}_{L^{\sharp,2}},\\
&\leq \frac{2c}{1-C\|F^{+}_{A}\|_{L^{,\sharp,2}}}C^{k-2}\|F^{+}_{A}\|_{L^{\frac{4}{3}}}\|F^{+}_{A}\|^{k-1}_{L^{\sharp,2}},\\
\end{split}
\end{equation}
We can provide the constants $\varepsilon$ sufficiently small and $C$ sufficiently large to ensures that  Equation (\ref{EA10}) is also satisfied for $j=k$.

The sequence $\zeta_{k}$ is Cauchy in $L^{\sharp,2}$, the limit $$\zeta:=\lim_{i\rightarrow\infty}\zeta_{i}$$ is a solution to (\ref{EA5}).  Then $L^{\frac{4}{3}}$estimate for $\zeta$,
\begin{equation}\label{EA7}
\|\zeta\|_{L^{\frac{4}{3}}}\leq\sum\|\xi_{k}\|_{L^{\frac{4}{3}}}\leq 2\|F_{A}^{+}\|_{L^{\frac{4}{3}}}.
\end{equation}
Observing that 
\begin{equation}\label{EA8}
d_{A}^{+}d_{A}^{+,\ast}u=d_{A}^{+}d_{A}^{+,\ast}G_{A}\zeta=\zeta
\end{equation}
so that the $L^{2}_{1}$ bound is obtained from
$$\|u\|_{L^{2}_{1}}\leq C\|d_{A}^{+,\ast}u\|_{L^{2}}\leq C\|\zeta\|_{L^{\frac{4}{3}}}\leq C\|F_{A}^{+}\|_{L^{\frac{4}{3}}},$$
where $C=C(g,\mu)$ is a positive constant.
\end{proof}
\section*{Acknowledgements}
I would like to thank the anonymous referee for  careful reading of my manuscript and helpful comments. I would like to thank Professor P. M. N. Feehan and Professor T. G. Leness for helpful comments regarding their article \cite{Feehan/Leness}. This work is supported by National Natural Science Foundation of China No. 11801539 and Postdoctoral Science Foundation of China No. 2017M621998, No. 2018T110616.

\end{document}